\documentclass[11pt]{article}
\usepackage{amsmath,amssymb}

\newtheorem{propo}{{\bf Proposition}}[section]
\newtheorem{coro}[propo]{{\bf Corollary}}
\newtheorem{lemma}[propo]{{\bf Lemma}} \newtheorem{theor}[propo]{{\bf
Theorem}} \newtheorem{ex}{{\sc Example}}[section]

\newenvironment{proof}{{\bf Proof.}}{$\Box$}

\begin{document}

\vspace*{1.0in}

\begin{center} THE INDEX COMPLEX OF A MAXIMAL SUBALGEBRA OF A LIE ALGEBRA 
\end{center}
\bigskip

\begin{center} DAVID A. TOWERS 
\end{center}
\bigskip

\begin{center} Department of Mathematics and Statistics

Lancaster University

Lancaster LA1 4YF

England

d.towers@lancaster.ac.uk 
\end{center}
\bigskip

\begin{abstract}
Let $M$ be a maximal subalgebra of the Lie algebra $L$. A subalgebra $C$ of $L$ is said to be a {\em completion} for $M$ if $C$ is not contained in $M$ but every proper subalgebra of $C$ that is an ideal of $L$ is contained in $M$. The set $I(M)$ of all completions of $M$ is called the {\em index complex} of $M$ in $L$. We use this concept to investigate the influence of the maximal subalgebras on the structure of a Lie algebra, in particular finding new characterisations of solvable and supersolvable Lie algebras.
\end{abstract}

\section{Introduction}
\medskip

Let $M$ be a maximal subalgebra of the Lie algebra $L$. A subalgebra $C$ of $L$ is said to be a {\em completion} for $M$ if $C$ is not contained in $M$ but every proper subalgebra of $C$ that is an ideal of $L$ is contained in $M$. The set $I(M)$ of all completions of $M$ is called the {\em index complex} of $M$ in $L$. This is analogous to the concept of the index complex of a maximal subgroup of a finite group as introduced by Deskins in \cite{des1}; this concept has since been further studied by a number of authors, including Ballester-Bolinches and Ezquerro (\cite{be}), Beidleman and Spencer (\cite{bs}), Deskins (\cite{des2}), Mukherjee (\cite{muck}), and Mukherjee and Bhattacharya (\cite{mb}). 
\par

There are a number of interesting results concerning the question of what certain intrinsic properties of the maximal subalgebras of a Lie algebra $L$ imply about the structure of $L$ itself. For example:
\begin{itemize}
\item[(1)] all maximal subalgebras are ideals of $L$ if and only if $L$ is nilpotent (see \cite{barnes1});
\item[(2)] all maximal subalgebras of $L$ are c-ideals of $L$ if and only if $L$ is solvable (see \cite{cideal});
\item[(3)] if $L$ is solvable then all maximal subalgebras have codimension one in $L$ if and only if $L$ is supersolvable (see \cite{barnes});
\item[(4)] $L$ can be characterised when its maximal subalgebras satisfy certain lattice-theoretic conditions, such as modularity (see \cite{var}).
\end{itemize}
The main purpose of this paper is to seek further such results; in particular, to find a characterisation of supersolvable Lie algebras amongst all Lie algebras, rather than just amongst the solvable ones as in (3) above. We are also seeking a characterisation of solvable Lie algebras in terms of the `size' of their maximal subalgebras. However, since, of course, the maximal subalgebras can have different codimensions in $L$, some other measure of their size is needed. The development of the theory follows closely that of its group-theoretic counterpart, but the proofs are usually different and stronger results can be obtained. 
\par

We define the {\em strict core} (resp. {\em core}) of a subalgebra $B \neq 0$ to be the sum of all ideals of $L$ that are proper subalgebras (resp. subalgebras) of $B$, and denote it by $k(B)$ or $k_L(B)$ (resp. $B_L$). Notice that the strict core can differ from the core even when $B$ is an ideal of $L$: for example, if $B$ is a one-dimensional ideal of $L$. It is easy to see that the strict core of any completion $C$ is a proper subalgebra of $C$. The subalgebra $C$ is then a completion of the maximal subalgebra $M$ of $L$ (that is, $C \in I(M)$) if $L = <M, C>$ and $k(C) \subseteq M$. Completions always exist, as the following lemma establishes.  
\bigskip

\begin{lemma}\label{l:index} If $M$ is a maximal subalgebra of $L$ then $I(M)$ is non-empty; in fact, $I(M)$ contains an ideal of $L$.
\end{lemma}
\begin{proof} Clearly the set of ideals of $L$ which do not lie in $M$ is a non-empty partially ordered set; choose $C$ to be a minimal element of this set. Then $L = M + C = <M, C>$ and $k(C) \subseteq M$, so $C \in I(M)$.
\end{proof}
\bigskip

In section two we study completions that are ideals and show that if $C$ and $D$ are two such completions of the same maximal subalgebra $M$ then $C/k(C) \cong D/k(D)$. This allows us to define the {\em ideal index} of $M$ in $L$. Next we characterise solvable and supersolvable Lie algebras in terms of the ideal index of their maximal subalgebras. We then consider completions $C$ that are ideals and for which $C/k(C)$ is abelian. A characterisation of solvable Lie algebras and of the solvable radical is given in terms of such completions.
\par

In section three maximal completions are studied. It is shown that over an algebraically closed field a Lie algebra is solvable if and only if every maximal subalgebra of $L$ has an abelian maximal completion. This is analogous to a result of Deskins (see \cite{des2}) for groups. The Lie algebraic result, however, is false if the underlying field is not algebraically closed, even if it has characteristic zero.
\par

The final section is devoted to subideal completions. A characterisation of supersolvable Lie algebras in terms of such completions is found that is analogous to a group-theoretic result of Ballester-Bolinches and Ezquerro (see \cite{be}). The Lie algebraic proof, however, is completely different, as the group theoretic result relies on properties that do not hold in the case of Lie algebras. A counter-example is also given to part of \cite[Corollary 2]{be}.
\par
Throughout, $L$ will denote a finite-dimensional Lie algebra over a field $F$. If $A$ and $B$ are subalgebras of $L$ for which $L = A + B$ and $A \cap B = 0$ we will write $L = A \oplus B$. The ideals $L^{(k)}$ of the derived series are defined inductively by $L^{(0)} = L$, $L^{(k+1)} = [L^{(k)},L^{(k)}]$  for $k \geq 0$; we also write $L^2$ for $L^{(1)}$. If $A$ is a subalgebra of $L$, the {\em centralizer} of $A$ in $L$ is $C_{L}(A) = \{ x \in L : [x, A] = 0\}$. 
\bigskip

\section{Ideal completions and the ideal index}
\medskip

If $C$ is an ideal of $L$ and $C \in I(M)$ we call $C$ an {\em ideal completion} of $L$. In this case $C/k(C)$ is a chief factor of $L$ which is avoided by $M$. Up to isomorphism it is uniquely determined by $M$, as is shown by the following result whose proof is based on that of \cite[Lemma 1]{bs}.
\bigskip 

\begin{theor}\label{t:isom} Let $C$ and $D$ be ideal completions of the maximal subalgebra $M$ of $L$. Then $C/k(C) \cong D/k(D)$.
\end{theor}
\begin{proof} Let $L$ be a Lie algebra of minimal dimension having a maximal subalgebra $M$ with two ideal completions $C$ and $D$ such that $C/k(C) \not \cong D/k(D)$. If $k(C) \cap k(D) \neq 0$ then by factoring out $k(C) \cap k(D)$ and using the minimality of $L$ we see that $C/k(C) \cong D/k(D)$, which is a contradiction. Hence $k(C) \cap k(D) = 0$. Now $C \cap k(D)$ is an ideal of $L$ and  $C \cap k(D) \subseteq C \cap M$, so $C \cap k(D) \subseteq k(C)$. It follows that $C \cap k(D) \subseteq k(C) \cap k(D) = 0$. Similarly $D \cap k(C) = 0$.
\par
Next we claim that $(C + k(D))/(k(C) + k(D))$ is a ideal completion to $M/(k(C) + k(D))$ in $L/(k(C) + k(D))$. For, certainly it is an ideal supplement. Suppose that $X/(k(C) + k(D))$ is an ideal of $L/(k(C) + k(D))$ properly contained in $(C + k(D))/(k(C) + k(D))$. Then $X \cap C$ is an ideal of $L$ properly contained in $C$, since $X \cap C = C$ implies that $C \subseteq X$ and $X = C + k(D)$. It follows that $X = X \cap (C + k(D)) = X \cap C + k(D) \subseteq M$, proving the claim. 
\par
Similarly $(D + k(C))/(k(C) + k(D))$ is an ideal completion to $M/(k(C) + k(D))$ in $L/(k(C) + k(D))$. But now if $k(C) + k(D) \neq 0$ the minimality of $L$ implies that $(C + k(D))/(k(C) + k(D)) \cong (D + k(C))/(k(C) + k(D))$. But $(C + k(D))/(k(C) + k(D)) \cong C/k(C)$ and $(D + k(C))/(k(C) + k(D)) \cong D/k(D)$ so this is a contradiction.
\par
We therefore have $k(C) + k(D) = 0$ and $C, D$ are minimal ideals of $L$. Then $L = M + C = M + D$ and $CD \subseteq C \cap D = 0$. But now $[M \cap C, L] = [M \cap C, M + D] \subseteq M \cap C$ so $M \cap C$ is an ideal of $L$. It follows that $M \cap C = 0$. Similarly $M \cap D = 0$. Thus
$$ C \cong \frac{C + D}{D} = \frac{(C + D) \cap (M + D)}{D} = \frac{(C + D) \cap M + D}{D} \cong (C + D) \cap M.
$$
Similarly, $D \cong (C + D) \cap M$, whence $C \cong D$. This contradiction establishes the result.
\end{proof}
\bigskip

The {\em ideal index} of a maximal subalgebra $M$ in $L$, $\eta(L:M)$, is the dimension of $C/k(C)$, where $C$ is an ideal completion of $M$ in $L$. 
\bigskip

\begin{coro}\label{c:unique} If $M$ is a maximal subalgebra of $L$ then $\eta(L:M)$ is well-defined.
\end{coro}
\bigskip

Next we establish how the ideal index behaves with respect to factor algebras. The following result, or rather its corollary, is analogous to \cite[Lemma 2]{bs}, though our proof is somewhat different.
\bigskip

\begin{propo}\label{p:fac} Let $M$ be a maximal subalgebra of $L$ and let $B$ be an ideal of $L$ with $B \subseteq M$. Let $C/B$ be an ideal completion of $M/B$ in $L/B$, put $k(C/B) = K/B$, and let $D$ be an ideal completion of $M$ in $L$. Then $C/K \cong D/k(D)$.
\end{propo}
\begin{proof} In view of Theorem \ref{t:isom} we can assume that $D \subseteq C$. We have that $(k(C) + B)/B$ is an ideal of $L$ which is inside $M/B$ and $C/B$, so $k(C) \subseteq K \cap D$. But also $K \cap D$ is an ideal of $L$ and $K \cap D \subseteq M \cap D$ so $K \cap D \subseteq k(C)$. Hence $k(C) = K \cap D$. But now $D \not \subseteq K$ since $K \subseteq M$, so $K + D = C$. This yields that $C/K = (K + D)/K \cong D/D\cap K = D/k(D)$.
\end{proof}
\bigskip

\begin{coro}\label{c:etafac} Let $M$ be a maximal subalgebra of $L$ and let $B$ be an ideal of $L$ with $B \subseteq M$. Then $\eta(L/B:M/B)= \eta(L:M)$.
\end{coro}
\begin{proof} From Proposition \ref{p:fac} we have that 
$$\eta(L/B:M/B)= \dim((C/B)/k(C/B)) = \dim(C/K)$$ 
$$= \dim(D/k(D)) =\eta(L:M).$$
\end{proof}
\bigskip

A subalgebra $B$ of a Lie algebra $L$ is called a {\em c-ideal} of $L$ if there is an ideal $C$ of $L$ such that $L = B + C$ and $B \cap C \leq B_L$, where $B_L$ is the largest ideal of $L$ contained in $B$. This concept was introduced and studied in \cite{cideal}. We have the following characterisation of maximal subalgebras that are c-ideals in terms of the ideal index.
\bigskip

\begin{theor}\label{t:cideal} Let $M$ be a maximal subalgebra of $L$. Then $M$ is a c-ideal of $L$ if and only if $\eta(L:M) = \dim(L/M)$.
\end{theor}
\begin{proof} ($\Rightarrow$) Let $L$ be a Lie algebra of smallest dimension having a maximal subalgebra $M$ which is a c-ideal but for which $\eta(L:M) \neq \dim(L/M)$. If $M_L \neq 0$ then $M/M_L$ is a c-ideal of $L/M_L$, by \cite[Lemma 2.1]{cideal}, and so $\eta(L/M_L:M/M_L) = \dim((L/M_L)/(M/M_L))$, by the minimality of $L$. But then $\eta(L:M) = \dim(L/M)$, by Corollary \ref{c:etafac}, a contradiction. If $M_L = 0$, by assumption there is an ideal $C$ of $L$ such that $L = M + C$ and $M \cap C \subseteq M_L = 0$. Since $M$ is a maximal subalgebra of $L$, $C$ is a minimal ideal of $L$. But then $\eta(L:M)  = \dim C = \dim(L/M)$. This contradiction yields the required implication.
\par
($\Leftarrow$) Now let $L$ be a Lie algebra of smallest dimension having a maximal subalgebra $M$ for which $\eta(L:M) = \dim(L/M)$ but $M$ is not a c-ideal of $L$. If $L$ is simple then $\eta(L:M) = \dim L$. But then $M = 0$ and $M$ is a c-ideal of $L$, a contradiction. If $M_L \neq 0$ then $M/M_L$ is a c-ideal of $L/M_L$, by the minimality of $L$, giving that $M$ is a c-ideal of $L$, by \cite[Lemma 2.1]{cideal}, a contradiction again. Thus we have that $L$ is a non-simple Lie algebra with $M_L = 0$. Let $B$ be a minimal ideal of $L$. Then $L = M + B$ and $\eta(L:M) = \dim B$. By our assumption then $\dim(L/M) = \dim B$. It follows that $M \cap B = 0$ and $M$ is a c-ideal of $L$. This contradiction again establishes the implication. 
\end{proof}
\bigskip

This yields the following characterisation of solvable Lie algebras in terms of the ideal index. The group-theoretic counterpart is \cite[Corollary, page 97]{bs}; its proof uses concepts that have no analogue in Lie algebras.
\bigskip

\begin{coro}\label{c:solv} The Lie algebra $L$ is solvable if and only if $\eta(L:M) = \dim(L/M)$ for all maximal subalgebras $M$ of $L$.
\end{coro}
\begin{proof} Simply combine Theorem \ref{t:cideal} and \cite[Theorem 3.1]{cideal}.
\end{proof}
\bigskip

With some restrictions on the underlying field the existence of a single solvable maximal subalgebra satisfying the above condition is sufficient to ensure that $L$ is solvable. the corresponding result for groups is \cite[Theorem 4]{bs}, though again its proof is completely different.
\bigskip

\begin{coro}\label{c:solv1} Let $L$ be a Lie algebra over a field $F$, where $F$ has characteristic zero or is algebraically closed of characteristic greater than 5. Then $L$ has a solvable maximal subalgebra $M$ with $\eta(L:M) = \dim(L/M)$ if and only if $L$ is solvable.
\end{coro}
\begin{proof} Simply combine Theorem \ref{t:cideal} and \cite[Theorems 3.2 and 3.3]{cideal}.
\end{proof}
\bigskip

We also have the following characterisation of supersolvable Lie algebras in terms of the ideal index.
\bigskip

\begin{coro}\label{c:etass} The Lie algebra $L$ is supersolvable if and only if $\eta(L:M) = 1$ for all maximal subalgebras $M$ of $L$.
\end{coro}
\begin{proof} Suppose first that $L$ is supersolvable and let $M$ be a maximal subalgebra of $L$. Then $L$ has codimension 1 in $L$, by \cite[Theorem 7]{barnes}, so $\eta(L:M) = 1$ by Corollary \ref{c:solv}.
\par
Now suppose that $L$ is any Lie algebra and that $M$ is a maximal subalgebra of $L$ with $\eta(L:M) = 1$. Then there is an ideal completion $C$ of $L$ with $\dim C/k(C) = 1$. Put $C = k(C) + Fx$. Then $L = M + Fx$ and $M$ has codimension 1 in $L$. But now $L$ is solvable, by Corollary \ref{c:solv}, and hence supersolvable, by \cite[Theorem 7]{barnes}.
\end{proof}
\bigskip

The {\em Frattini subalgebra} of $L$, $F(L)$, is the intersection of all of the maximal subalgebras of $L$. The {\em Frattini ideal}, $\phi(L)$, of $L$ is $F(L)_L$. The group-theoretic counterpart of the next result, which has a different proof, is \cite[Theorem 6]{bs}
\bigskip

\begin{theor}\label{t:ss} If $L$ has a supersolvable maximal subalgebra $M$ with $\eta(L:M) = 1$ and $N(L) \not \subseteq M$ then it is supersolvable.
\end{theor}
\begin{proof} Suppose first that $L$ is $\phi$-free. Then $N(L) = Asoc(L)$, by Theorem 7.4 of \cite{frat}, so there is a minimal abelian ideal $A$ of $L$ with $A \not \subseteq M$. Clearly $L = A \oplus M$. Moreover, $\dim A = 1$ as in Corollary \ref{c:etass} above, and so $L$ is supersolvable.
\par
If $L$ is not $\phi$-free then $\eta(L/\phi(L):M/\phi(L)) = 1$, by Corollary \ref{c:etafac}, and $N(L/\phi(L)) = N(L)/\phi(L)$, by \cite[Theorem 6.1]{frat}, so $L/\phi(L)$ satisfies the hypotheses of this theorem. It follows from the paragraph above that $L/\phi(L)$ is supersolvable. But then $L$ is supersolvable, by \cite[Theorem 6]{barnes}. 
\end{proof}
\bigskip

We say that the maximal subalgebra $M$ of $L$ has an {\em abelian ideal completion} if it has an ideal completion $C$ such that $C/k(C)$ is abelian. Then we have the following result, which is more straightforward to prove than its analogue in group theory: \cite[Theorem, page 237]{des2}.
\bigskip

\begin{theor}\label{t:ab} The Lie algebra $L$ is solvable if and only if every maximal subalgebra of $L$ has an abelian ideal completion.
\end{theor}
\begin{proof} Suppose first that $L$ is solvable and that $M$ is a maximal subalgebra of $L$. Then there exists $k \geq 2$ such that $L^{(k)} \subseteq M$ but $L^{(k-1)} \not \subseteq M$. Put $C = L^{(k-1)}$. Then $L = M + C$ so $C$ is an ideal completion of $M$. Also, $L^{(k)} \subseteq k(C)$, so $C/k(C)$ is abelian.
\par
Suppose now that every maximal subalgebra of $L$ has an abelian ideal completion, let $M$ be a maximal subalgebra of $L$ and let $C$ be an abelian ideal completion of $M$. Then $[C, M \cap C] \subseteq C^2 \subseteq k(C) \cap C \subseteq M \cap C$, so $M \cap C$ is an ideal of $L$. We infer that $M \cap C \subseteq M_L$ and hence that $M$ is a c-ideal of $L$. The result now follows from \cite[Theorem 3.1]{cideal}.  
\end{proof}
\bigskip

Put $F^*(L)$ equal to the intersection of all maximal subalgebras of $L$ which have no abelian ideal completion (with $F^*(L) = L$ if no such maximal subalgebras exist), and let $\phi^*(L) = F^*(L)_L$. Then we have the following characterisation of the radical of a Lie algebra, whose counterpart in group theory is \cite[Theorem B, page 238]{des2}; again our proof is easier.
\bigskip

\begin{theor}\label{t:rad} Let $L$ be any Lie algebra. Then $\phi^*(L) = R(L)$, the solvable radical of $L$.
\end{theor}
\begin{proof} Suppose first that $M$ is a maximal subalgebra of $L$ with $R(L) \not \subseteq M$. Let $C$ be minimal in the set of ideals of $L$ that are inside $R(L)$ but not in $M$. Then $C$ is an ideal completion of $M$ and $C/k(C)$ is a minimal solvable ideal and so abelian. Thus $R(L) \subseteq \phi^*(L)$.
\par
Next suppose that $R(L) = 0$. We wish to show that $\phi^*(L) = 0$. Suppose not, and let $A$ be a minimal ideal of $L$ with $A \subseteq \phi^*(L)$. Since $\phi(L) = 0$ (see \cite[Theorem 6.5]{frat}), there is a maximal subalgebra $M$ of $L$ with $A \not \subseteq M$, so $L = M + A$ and $M_L \cap A = 0$. 
\par
Now $F^*(L) \not \subseteq M$, so $M$ has an abelian ideal completion $C$. Clearly $L = M + C$ and $M_L \cap C = k(C)$. If $A \subseteq C$ then $A^2 \subseteq C^2 \cap A \subseteq k(C) \cap A \subseteq M_L \cap A = 0$ and $A$ is abelian, contradicting the fact that $R(L) = 0$. It follows that $A \cap C = 0$ whence $[A, C] = 0$.
\par
Put $K = \{ x \in L : [x,C] \subseteq M_L \}$. Then $M_L + C \subseteq K$. Furthermore, $[M \cap K, C] \subseteq [K, C] \subseteq M_L \cap K \subseteq M \cap K$, so $M \cap K$ is an ideal of $L$. This yields that $M \cap K \subseteq M_L$ from which $K = M \cap K + C \subseteq M_L + C$. We therefore have that $K = M_L + C$ and $M_L + A \subseteq M_L + C$. But now
$$
A \cong \frac{A}{M_L \cap A} \cong \frac{M_L + A}{M_L} \subseteq \frac{M_L + C}{M_L} \cong \frac{C}{M_L \cap C} = \frac{C}{k(C)}
$$
and so $A$ is abelian. This is impossible, since $R(L) = 0$, so $\phi^*(L) = 0$.
\end{proof}
\bigskip

\section{Maximal completions}
\medskip

The set $I(M)$ is partially ordered by set inclusion; call maximal elements of $I(M)$ {\em maximal completions} of $M$ in $L$. Clearly every ideal completion of $M$ in $L$ is a maximal completion of $M$ in $L$, but the converse is not true in general: for example, if $L$ is the two-dimensional non-abelian Lie algebra with basis $x, y$ in which $[x, y] = y$ and $M = F(x + y)$, then $Fx$ is a maximal completion of $M$ in $L$ but is not an ideal of $L$.
\par

Here we are seeking an analogue to a result of Deskins (see \cite{des2}): namely, that if a Lie algebra has a maximal completion $C$ with $C/k(C)$ abelian then it has an abelian ideal completion. As we shall see, this result holds only with conditions on the underlying field. First we consider the structure of Lie algebras with a maximal abelian subalgebra. A Lie algebra $L$ is {\em completely solvable} if $L^{(1)}$ is nilpotent.
\bigskip 

\begin{propo}\label{p:compsolv} Let $L$ be a completely solvable Lie algebra. Then $L$ has a maximal subalgebra $M$ that is abelian if and only if either
\begin{itemize}
\item[(i)] $L$ has an abelian ideal of codimension one in $L$; or
\item[(ii)] $L^{(2)} = \phi(L) = Z(L)$, $L^{(1)}/L^{(2)}$ is a chief factor of $L$, and $L$ splits over $L^{(1)}$.
\end{itemize}
\end{propo}
\begin{proof} Suppose first that $L$ has a maximal subalgebra $M$ that is abelian. If $M$ is an ideal of $L$ we have case (i). So suppose that $M$ is self-idealising, in which case it is a Cartan subalgebra of $L$. Now $L^{(2)} \subseteq \phi(L) \subseteq M$, by \cite[Theorem 6.5]{frat}. If $S$ is a subalgebra of $L$ denote by $\bar{S}$ its image under the canonical homomorphism onto $L/L^{(2)}$. Then $\bar{M}$ is a Cartan subalgebra of $\bar{L}$ and $\bar{L}$ has a Fitting decomposition $\bar{L} = \bar{M} \oplus \bar{L}_1$. Now $\bar{L}_1 \subseteq \bar{L}^{(1)} = \overline{L^{(1)}}$ which is abelian, so $\bar{L}_1$ is an ideal of $\bar{L}$. Moreover, since $\bar{M}$ is a maximal subalgebra of $\bar{L}$, $\bar{L}_1$ is a minimal ideal of $\bar{L}$ and $\overline{L^{(1)}} = \bar{L}_1$. It follows that $L^{(1)}/L^{(2)}$ is a chief factor of $L$. Clearly $\phi(\bar{L}) = 0$, whence $\phi(L) = L^{(2)}$. Also $L = M + L^{(1)}$, so letting $C$ be a subspace of $M$ such that $M = C \oplus (M \cap L^{(1)})$ we see that $C$ is a subalgebra of $L$ and $L$ splits over $L^{(1)}$. Finally, $[M,L^{(2)}] \subseteq M^{(1)} = 0$, so $M \subseteq C_L(L^{(2)})$. Since $M$ is a self-idealising maximal subalgebra of $L$ and $C_L(L^{(2)})$ is an ideal of $L$, we have $C_L(L^{(2)}) = L$, whence $L^{(2)} = Z(L)$.
\par
Consider now the converse. If (i) holds the converse is clear. So suppose that (ii) holds. Then $L = C \oplus L^{(1)}$ where $C$ is an abelian subalgebra of $L$. Put $M = C + L^{(2)}$, so $M$ is clearly abelian. Let $M \subseteq T \subseteq L$. Then $\bar{M} \subseteq \bar{T} \subseteq \bar{L}$. But $\bar{L} = \bar{M} \oplus \overline{L^{(1)}}$ and $\overline{L^{(1)}}$ is a minimal abelian ideal of $\bar{L}$. So $\bar{M} \neq \bar{T}$ implies that $\bar{T} = \bar{L}$. It follows that $M$ is a maximal subalgebra of $L$.  
\end{proof} 
\bigskip

\begin{propo}\label{p:maxab} Let $L$ be a Lie algebra over an algebraically closed field $F$ of any characteristic. Then $L$ has a maximal subalgebra $M$ that is abelian if and only if either
\begin{itemize}
\item[(i)] $L$ has an abelian ideal of codimension one in $L$; or
\item[(ii)] $L^{(1)}$ has dimension one and $\phi(L) = 0$.
\end{itemize}  In either case, $L$ is completely solvable. 
\end{propo}
\begin{proof} Suppose first that $L$ has a maximal subalgebra $M$ that is abelian. If $M$ is an ideal of $L$ we have case (i). So suppose that $M$ is self-idealising, in which case it is a Cartan subalgebra of $L$. Let $L = M \oplus L_{1}(M)$ be the Fitting decomposition of $L$ relative to $M$. Then $\{({\rm ad}\,m)\mid_{L_{1}(M)}: m \in M\}$ is a set of simultaneously triangulable linear mappings. So, there exists $0 \neq b \in L_{1}(M)$ such that $[m, b] 
= \alpha(m) b$ for every $m \in M$, where $\alpha(m) \in F$. Then we 
have that $M + Fb$ is a subalgebra of $L$ strictly containing $M$, whence $M \oplus Fb = L$. But now $L^{(1)} = Fb$, $\phi(L) \subseteq M \cap L^{(1)} = 0$ and we have case (ii).
\par
Consider now the converse. If (i) holds the converse is clear. So suppose that (ii) holds. Then $L = L^{(1)} \oplus M$ for some subalgebra $M$ of $L$, by \cite[Lemma 7.2]{frat}. Clearly $M$ is an abelian maximal subalgebra of $L$. 
\end{proof}
\bigskip

\begin{theor}\label{t:maxcomp} Let $L$ be a Lie algebra over an algebraically closed field and let $M$ be a maximal subalgebra of $L$ with a maximal completion $C$ such that $C/k(C)$ is abelian. Then $M$ has an abelian ideal completion in $L$.
\end{theor}
\begin{proof} Let $L$ be a Lie algebra of smallest dimension for which the result is false, and let $M$ be a maximal subalgebra of $L$ with an abelian maximal completion $C$ but no abelian ideal completion. If $L$ is simple then $k(C) = 0$ and the maximality of $C$ in $I(M)$ implies that $C = L$, which is impossible.
\par
If $M$ is an ideal of $L$ choose $A$ to be minimal in the set of ideals of $L$ (not necessarily proper) that are not contained in $M$. Then $A$ is an ideal completion of $M$ in $L$, $A \cap M = k(A)$ and $A/k(A) \cong L/M$ which is one dimensional and hence abelian. 
\par
Suppose now that $M_L = 0$. Then $k(C) = 0$ and we can assume that $C$ is not an ideal of $L$. Now $L$ contains a subalgebra $D$ in which $C$ is a maximal subalgebra, and $D$ is solvable, by Proposition \ref{p:maxab}. The result certainly holds if $L$ is solvable, by Theorem \ref{t:ab}, so we can assume that $D \neq L$. It follows from the maximality of $C$ in $I(M)$ that $D \notin I(M)$. Since $<M, D> = L$ we must have that $k(D) \not \subseteq M$. Thus there is a proper subalgebra $K$ of $D$ such that $K$ is an ideal of $L$ and $L = M + K$. Let $A$ be a minimal non-trivial ideal of $L$ inside $D$. Then $M + A = L$, $k(A) = 0$ and $A$ is abelian because $D$ is solvable. This means that $A$ is an abelian ideal completion of $M$. 
\par
Suppose next that $M_L = K \neq 0$ and $K \neq k(C)$, and consider the subalgebra $K + C \neq C$. Then $<M, K + C>\,\, = L$, but $K + C \notin I(M)$ because of the maximality of $C$ in $I(M)$. It follows that $k(K + C) \not \subseteq M$, and so the collection of ideals of $L$ inside $K + C$ but not in $M$ is non-empty. Let $A$ be a minimal element of this set. Clearly $A$ is an ideal completion of $M$ in $L$. Also $(K + C)/K \cong C/C \cap K$ is abelian, since $C/k(C)$ is abelian and $k(C) \subseteq C \cap K$, so $A/k(A) = A/K \cap A \cong (K + A)/K$ is abelian, because $A \subseteq K + C$.
\par
Finally suppose that $M_L = K \neq 0$ and $K = k(C)$. If $S$ is a subalgebra of $L$ denote by $\bar{S}$ its image under the canonical homomorphism onto $L/K$. Then $\bar{C}$ is a completion of $\bar{M}$ in $\bar{L}$ and $\bar{M}$ is core-free. If $\bar{C}$ is a maximal element of $I(\bar{M})$, then, by the paragraph above, $\bar{M}$ has an abelian ideal completion $\bar{A}$ in $\bar{L}$. Then $A$ is an ideal completion of $M$ in $L$ and $A/k(A) = A/K \cap A \cong (A + K)/K \cong \bar{A}$ is abelian.
\par
If $\bar{C}$ is not a maximal element of $I(\bar{M})$ then let $\bar{D}$ be minimal amongst those elements of $I(\bar{M})$ which contain $C$ properly. Clearly $C$ is a maximal subalgebra of $D$, so $D$ is not a completion of $M$ in $L$, by the maximality of $C$. Hence $k(D) \not \subseteq M$. Choose $A$ to be a minimal element in the (non-empty) collection of ideals of $L$ lying in $D$ but not in $M$. Then $A$ is an ideal completion of $M$. Also, $D/K$ is solvable, by Proposition \ref{p:maxab}, since $C/K$ is maximal in $L/K$. Moreover, $L/K \supseteq (A + K)/K \cong A/A \cap K$ and $A \cap K = k(A)$, so the chief factor $A/k(A)$ of $L$ is solvable, and thus abelian.
\end{proof}
\bigskip

\begin{coro}\label{c:maxcomp} Let $L$ be a Lie algebra over an algebraically closed field. Then $L$ is solvable if and only if every maximal subalgebra of $L$ has a maximal completion $C$ in $L$ with $C/k(C)$ abelian.
\end{coro}
\begin{proof} This follows from Theorems \ref{t:ab} and \ref{t:maxcomp}.
\end{proof}
\bigskip 

Note that Proposition \ref{p:maxab} and Theorem \ref{t:maxcomp} do not hold when the underlying field $F$ is not algebraically closed, even if it has characteristic zero, as the following example shows.
\par

\begin{ex} Let $S$ have basis $e_1, e_2, e_3$ with $[e_1,e_2] = -[e_2,e_1] = e_3$, $[e_2,e_3] = -[e_3,e_2] = -e_1$, $[e_3,e_1] = -[e_1,e_3] = e_2$, all other products being zero (so $L$ is three-dimensional non-split simple), let $\bar{S}$ be an isomorphic copy of $S$ and denote the image of $s \in S$ in $\bar{S}$ by $\bar{s}$. Put $L = S \oplus \bar{S}$ with $[S, \bar{S}] = 0$. 
\end{ex}
Every maximal subalgebra of $S$ is one dimensional, and so abelian, showing that Proposition \ref{p:maxab} does not hold. Also, it is easy to check that the diagonal subalgebra $M = \{ x \in L : x = s + \bar{s} \hbox{ for some } s \in S \}$ is maximal in $L$ and that $C = Fe_1 + F\bar{e_1}$ is a maximal abelian completion of $M$ in $L$. However, $M$ has no maximal abelian ideal completion in $L$.
\bigskip

\section{Subideal completions}
\medskip

When the underlying field is algebraically closed of characteristic zero, the following characterisation of supersolvable Lie algebras in terms of maximal completions follows easily from Theorem \ref{t:maxcomp}, Corollary \ref{c:etass} and Lie's Theorem.
\bigskip

\begin{propo}\label{p:supsolv} Let $L$ be a Lie algebra over an algebraically closed field of characteristic zero. Then $L$ is supersolvable if and only if every maximal subalgebra of $L$ has a maximal completion $C$ with $\dim C/k(C) = 1$.
\end{propo}
\bigskip

The above result does not hold, however, over every field of characteristic zero, as the following example shows.
\bigskip

\begin{ex} Let $L = Fa \oplus S$, where $Fa$ is a one-dimensional abelian ideal of $L$ and $S$ is a three-dimensional non-split simple ideal of $L$. Then the maximal subalgebras of $L$ are $S$ and the subalgebras of the form $Fa + Fx$ where $x \in S$. For the former a maximal completion is $C = Fa$; for the latter a maximal completion is $C = Fa + Fy$ where $y \in S \setminus Fx$. In either case $\dim C/k(C) = 1$, but $L$ is not supersolvable.
\end{ex} 
\bigskip

So, in order to find another characterisation of supersolvable Lie algebras over a general field in terms of the index complex we follow \cite{be}. If $C$ is a subideal of $L$ and $C \in I(M)$ we call $C$ a {\em subideal completion} of $L$. Unlike ideal completions there is no numerical invariant associated with subideal completions, as the following example shows.
\par

\begin{ex} Let $L$ be the Lie algebra with basis $a, b, c,d$ and products $[a,b] = - [b,a] = c, [b,c] = - [c,b] = d$, all other products being zero. Then $L$ is nilpotent and so all of its subalgebras are subideals of $L$. Put $M = Fb + Fc + Fd$. Then $M$ is a maximal subalgebra of $L$ and it is easy to check that $C_1 = Fa$ and $C_2 = Fa + Fc$ are both completions of $M$ in $L$. However, $k(C_1) = k(C_2) = 0$, so $\dim(C_1/k(C_1)) = 1$, whereas $\dim(C_2/k(C_2)) = 2$.
\end{ex}
\bigskip

In the following we will need subideal completions with an extra property. Let $M$ be a maximal subalgebra of the Lie algebra $L$ and let $S(M) = \{ C \in I(M) : C \hbox{ is a subideal of } L \hbox{ and } L = M + C \}$. Clearly, every ideal completion of $M$ in $L$ is in $S(M)$, so $S(M)$ is non-empty. 
\par

Then we have the following characterisation of supersolvable Lie algebras in terms of such completions which is analogous to \cite[Theorem 1]{be}. To use a similar argument to theirs we would require an underlying field of characteristic zero, as it relies on the Baer radical being a nilpotent ideal (see \cite{as}). The argument given below is independent of the underlying field.
\bigskip

\begin{theor}\label{t:subideal} The Lie algebra $L$ is supersolvable if and only if every maximal subalgebra of $L$ has an element $C \in S(M)$ with $\dim C/k(C) = 1$.
\end{theor}
\begin{proof} Suppose first that every maximal subalgebra $M$ of $L$ has an element $C \in S(M)$ with $\dim C/k(C) = 1$. It is clear that $L/\phi(L)$ satisfies the same hypotheses. Moreover, if $L/\phi(L)$ is supersolvable then so is $L$, by \cite[Theorem 6]{barnes}, so we can assume that $L$ is $\phi$-free. Since $k(C) \subseteq M$, we have that $M$ has codimension one in $L$. It follows from \cite{codone} that $L \cong R \oplus S$, where $R$ is a supersolvable ideal of $L$ and $S$ is a three-dimensional simple ideal of $L$. 
\par

Suppose that $S \neq 0$, let $M$ be a maximal subalgebra of $L$ containing $R$, let $C \in S(M)$ with $\dim C/k(C) = 1$ and let $\pi : L \rightarrow S$ be the projection homomorphism from $L$ onto $S$. Then $\pi(C)$ is a subideal of $S$, so $\pi(C) = S$ or $0$. Since $\pi(k(C)) = 0$ the former is impossible. The latter implies that $C \subseteq R$ which is also impossible. It follows that $S = 0$ and $L$ is supersolvable.
\par

Conversely, suppose that $L$ is supersolvable. Then $L$ has an ideal completion $C$ (and so $C \in S(M)$) with $\dim C/k(C) = 1$ by Corollary \ref{c:etass}. 
\end{proof} 
\bigskip

A class ${\mathcal H}$ of finite-dimensional solvable Lie algebras is called a {\em homomorph} if it contains, along with an algebra $L$, all epimorphic images of $L$. The following result is a straightforward adaptation of \cite[Proposition 1]{be}. We include the proof for the convenience of the reader.
\bigskip

\begin{propo}\label{p:hom} Let ${\mathcal H}$ be a homomorph that is closed under ideals, let $M$ be a maximal subalgebra of $L$ and let $N$ be an ideal of $L$ such that $N \subseteq M$. If $C$ is a maximal (respectively, subideal) completion of $M$ in $L$ with $C/k(C) \in {\mathcal H}$, there is a maximal (respectively, subideal) completion $C^*$ of $M$ in $L$ such that $N \subseteq C^*$ and $C^*/k(C^*) \in {\mathcal H}$. 
\end{propo}
\begin{proof} Assume that $M$ has a maximal completion $C$ in $L$ with $C/k(C) \in {\mathcal H}$. If $N \subseteq C$ we can take $C^* = C$, so assume that $N \not \subseteq C$. Since $C$ is a maximal completion of $M$ and $C \subset N + C$ we have $k(N + C) \not\subseteq M$. Hence there is a chief factor $C^*/P$ of $L$ such that $N + k(C) \subseteq P \subset C^* \subseteq k(N + C)$ and $L = M + C^*$. Thus $C^*$ is a maximal completion of $M$ and 
$$ \frac{k(N + C)}{N + k(C)} \hbox{ is an ideal of } \frac{N + C}{N + k(C)} \in {\mathcal H}, \hbox{ whence } \frac{k(N + C)}{N + k(C)} \in {\mathcal H}.$$
Consequently, $C^*/P \in {\mathcal H}$. Since $k(C^*) = P$, $C^*$ is a maximal completion of $M$ such that $N \subseteq C^*$.
\par

For subideal completions the argument is similar.
\end{proof}
\bigskip

Let ${\mathcal H}$ be as in the above Proposition and let $I({\mathcal H}) = \{L :$ for each maximal subalgebra $M$ of $L$ there exists $C \in I(M)$, maximal in $I(M)$, with $C/k(C) \in {\mathcal H}\}$, $S({\mathcal H}) = \{L :$ for each maximal subalgebra $M$ of $L$ there exists $C \in S(M)$ with $C/k(C) \in {\mathcal H}\}$. Then, as a basis for induction arguments, it is claimed in \cite[Corollary 2]{be} that, in respect of the corresponding concepts for groups, $I({\mathcal H})$ and $S({\mathcal H})$ are saturated homomorphs. However, if $K = \{e,a,b,c\}$ is Klein's 4-group and ${\mathcal H}$ is the homomorph of groups of order one, then it is easy to see that $K \in I({\mathcal H})$, whereas $K/\{e,a\} \not \in I({\mathcal H})$, since it has order two and is the only maximal completion for its trivial maximal subgroup. Less trivial examples are also easy to construct.
\par

In the case of Lie algebras $S({\mathcal H})$ is also a saturated homomorph, but $I({\mathcal H})$ is not. For, if $L$ is as in Example 4.1 above and ${\mathcal H}$ is the homomorph of all abelian Lie algebras, then $L \in I({\mathcal H})$ but $L/Fa \cong S \not \in I({\mathcal H})$, since $S$ itself is the only maximal completion for any of its maximal subalgebras and $S/k(S) \cong S$ is not abelian.

\end{document}